\documentclass[11pt]{amsart}

\author{Matthew D. Kvalheim}
\address{Department of Mathematics and Statistics, University of Maryland, Baltimore County, MD, USA} 
\email{kvalheim@umbc.edu}

\author{Eduardo D. Sontag}
\address{Departments of Electrical and Computer Engineering and Bioengineering, and affiliate of Departments of Mathematics and
	Chemical Engineering, Northeastern University, Boston, MA, USA}
\email{sontag@sontaglab.org, e.sontag@northestern.edu, eduardo.sontag@gmail.com}

\usepackage{amsthm}
\usepackage{amsmath}
\usepackage{xcolor}
\usepackage{mathtools}

\newcommand{\N}{\mathbb{N}}
\newcommand{\R}{\mathbb{R}}
\newcommand{\Sph}{\mathbb{S}}

\theoremstyle{definition}

\newtheorem{Th}{Theorem}

\newtheorem{Prop}{Proposition}

\newtheorem*{Quest-non}{Question}

\newtheorem{Rem}{Remark}

\usepackage{marvosym}
\usepackage[
hypertexnames,%
citecolor=blue,%
colorlinks=true,%
linkcolor=red%
]{hyperref}
\usepackage[all]{hypcap}

\title{Global linearization of asymptotically stable systems without hyperbolicity}

\begin{document}

\begin{abstract}
	We give a proof of an extension of the Hartman-Grobman theorem to nonhyperbolic but asymptotically stable equilibria of vector fields.
	Moreover, the linearizing topological conjugacy is (i) defined on the entire basin of attraction if the vector field is complete, and (ii) a $C^{k\geq 1}$-diffeomorphism on the complement of the equilibrium if the vector field is $C^k$ and the underlying space is not $5$-dimensional.
	We also show that  the $C^k$ statement in the $5$-dimensional case is equivalent to the $4$-dimensional smooth Poincar\'{e} conjecture.
\end{abstract}

\maketitle

\section{Introduction}\label{sec:intro}

Consider a nonlinear system of ordinary differential equations
\begin{equation}\label{eq:ode}
	\dot{x}(t)=f(x(t)),	
\end{equation}
where $f$ is a vector field on an $n$-dimensional manifold $M$.
If $f\in C^1$ and $x_*\in M$ is a hyperbolic equilibrium for $f$, the Hartman-Grobman theorem guarantees existence of continuous local coordinates on a neighborhood of $x_*$ in which the \emph{nonlinear} dynamics \eqref{eq:ode} become \emph{linear} \cite{hartman1960lemma,grobman1959homeomorphism}.

In this paper, we provide a proof of an extension of the Hartman-Grobman theorem to nonhyperbolic but asymptotically stable equilibria.
We also assume only that $f$ is continuous and uniquely integrable (e.g. locally Lipschitz).
But if additionally $f\in C^k$ and $n\neq 5$, we construct continuous linearizing coordinates that are $C^k$ away from $x_*$.
Finally, if $f$ is complete, we prove the existence of globally linearizing coordinates on the entire basin of attraction of $x_*$.
In fact, the linear dynamics can be taken to be $\dot{y}=-y$, in which case the coordinates transform the nonlinear dynamics into the negative gradient flow of the convex function $y\mapsto \|y\|^2/2$.

To our knowledge, the first trace of these extensions appeared in work of Coleman, who proved a theorem equivalent to our local result for $k=0$ assuming $x_*$ has a Lyapunov function with level sets homeomorphic to spheres \cite{coleman1965local,coleman1966addendum} (see \cite[p.~247]{wilson1978reformulation}).
Our proof of the global result fills in details of one previously sketched by Gr\"une, the second author, and Wirth for $k=0$ and $n\neq 4$, and for general $k$ and $n\neq 4, 5$ \cite[remark in p.~133]{grune1999asymptotic}.
Using the same techniques and Perelman's solution to the $3$-dimensional Poincar\'{e} conjecture \cite{perelman2002entropy,perelman2003finite,perelman2003ricci}, we extend this global result to all $n$ for $k=0$, and to $n\neq 5$ for general $k$.
In fact, we prove that whether the $C^{k\geq 1}$ linearization results hold for $n=5$ is \emph{equivalent} to the $4$-dimensional smooth Poincar\'{e} conjecture (still open).
Jongeneel recently proved other interesting stability results using Perelman's solution \cite{jongeneel2024asymptotic}.
We also note connections to classical dynamical systems results showing that systems with globally asymptotically stable equilibria admit $C^0$ Lyapunov functions $V$ with exponential decrease $V(x(t)) = e^{-t}V(x(0))$ 
\cite[Chapter V.2]{BhatSzeg70}.
Furthermore, in the special case of linear systems, it has long been known that each linear system is equivalent, under a continuous coordinate change, to the system $\dot y = -y$ of the same dimension \cite{Arnold92}.

The local linearization result is Theorem~\ref{th:local-lin}, and the global result is Theorem~\ref{th:global-lin}.
While some techniques  in the hyperbolic case establish global linearizations by extending local ones \cite{lan2013linearization, eldering2018global,kvalheim2021existence}, in this paper we prove the global result directly (and under no hyperbolicity assumptions) and the local result as a consequence.
In contrast to the classical hyperbolic results, we assume stability, since otherwise there are no natural global Lyapunov functions to guide the constructions.

It seems worth noting that Theorems~\ref{th:local-lin}, \ref{th:global-lin} give new existence results for targets of algorithms like extended Dynamic Mode Decomposition \cite{williams2015data} studied by applied Koopman operator theorists \cite{budisic2012appliedkoopmanism,mezic2020spectrum,brunton2022modern}.
Such algorithms seek to compute $N\geq \dim M$ linearizing ``observables'' through which nonlinear systems appear linear. Theorems~\ref{th:local-lin}, \ref{th:global-lin} (see also Remark~\ref{rem:eigenfunctions-independent}) give existence results in the case $N=\dim M$, a case of arguably practical importance \cite{haller2024data} complementing recent (non)existence results for the case $N>\dim M$ \cite{belabbas2023super,liu2023non,kvalheim2024linearizability,arathoon2023koopman,liu2024properties,ko2024minimum}.

\section{Results}\label{sec:results}

We begin with the local linearization result, whose statement refers to the initial value problem
\begin{equation}\label{eq:ivp}
	\dot{x}(t)=f(x(t)),\quad x(0)=x_0.	
\end{equation}

\begin{Th}\label{th:local-lin}
	Let $x_*$ be an asymptotically stable equilibrium for a uniquely integrable continuous vector field $f$ on an $n$-dimensional $C^\infty$ manifold $M$.
	There is an open neighborhood $U\subset M$ of $x_*$ such that, for any Hurwitz matrix $A\in \R^{n\times n}$, there is a topological embedding $h\colon U\to\R^n$ such that, for all $x_0\in U$, the maximal solution of \eqref{eq:ivp} satisfies $x(t)\in U$ and
	\begin{equation}\label{eq:th:local-lin}
		\textnormal{$x(t)=h^{-1}(e^{At}h(x_0))$ \quad for all $t\geq 0$.}	
	\end{equation}
	Moreover, if $n\neq 5$ and $f\in C^k$ with $k\in \N_{\geq 1}\cup \{\infty\}$,  then there exists such an $h$ that additionally restricts to a $C^k$ embedding $U\setminus \{x_*\}\to \R^n\setminus \{0\}$.
\end{Th}

An advantage of Theorem~\ref{th:local-lin} is that it does not assume completeness of $f$.
On the other hand, the global linearization result below is formulated more optimally for smoothness, since the vector field generating a $C^k$ flow need only be $C^{k-1}$ in general.

\begin{Th}\label{th:global-lin}
	Let $x_*$ be an asymptotically stable equilibrium with basin of attraction $B$ for the flow $\Phi$ of a complete uniquely integrable continuous vector field on an $n$-dimensional $C^\infty$  manifold $M$.
	For any Hurwitz matrix $A\in \R^{n\times n}$ there is a homeomorphism $h\colon B\to \R^n$ satisfying
	\begin{equation}\label{eq:th:global-lin}
		\textnormal{$\Phi^t|_B = h^{-1} \circ e^{At}\circ h$ \quad for all \quad $t\in \R$.}
	\end{equation}
	Moreover, if $n\neq 5$ and $\Phi\in C^k$ with $k\in \N_{\geq 1}\cup \{\infty\}$,  then there exists such an $h$ that additionally restricts to a $C^k$-diffeomorphism $B\setminus \{x_*\}\to \R^n\setminus \{0\}$.
\end{Th}

\begin{Rem}\label{rem:gradient}
	In particular, consider $A=-I_{n\times n}$ in the above theorems.
	Then as noted in \S \ref{sec:intro}, $h$ transforms the nonlinear dynamics into the negative gradient flow of the dynamics $\dot{y}=-y$ of the convex function $y\mapsto \|y\|^2/2$.
	This is distinct from the fact that a Riemannian metric always exists making the vector field a gradient on the complement of the equilibrium within the basin of attraction \cite[Thm~1]{barta2012lyapunov} in all dimensions (but for $n\neq 5$ this directly follows from the above theorems).
	The same explanation shows that $h$ transforms the nonlinear dynamics into a ``contractive system'' \cite{sontag10yamamoto}.
	There are generally no explicit expressions for the transformation $h$, except in very simple cases. For example, the one-dimensional system $\dot x = -x^3$ is transformed into $\dot y = -y$ by means of the transformation $y =  h(x) = e^{-\frac{1}{2x^2}}$.
\end{Rem}

\begin{Rem}\label{rem:eigenfunctions-independent}
	By choosing $A$ to be diagonal, the conclusion of Theorem~\ref{th:global-lin} furnishes $n$-tuples $(\psi_1,\ldots,\psi_n)$ of continuous real eigenfunctions of the Koopman operator such that the mapping $(\psi_1,\ldots,\psi_n)\colon B\to \R^n$ is a homeomorphism.
	Indeed, we may consider the functions $\psi_i = h_i$, where $h$ is the homeomorphism from our theorem statement with components $h_1,\ldots, h_n$.
\end{Rem}

\begin{Rem}\label{rem:hyperbolic}
	If the equilibrium $x_*$ is hyperbolic, then the dimensional restriction $n\neq 5$ in Theorems~\ref{th:local-lin}, \ref{th:global-lin} is not needed since then there exists a smooth Lyapunov function having nondegenerate quadratic (Morse) singularity at $x_*$.
	The nonzero level sets of such a Lyapunov function are diffeomorphic to the standard sphere of the appropriate dimension, so any such level set can be used instead of $L$ in the proof of Theorem~\ref{th:global-lin}.
\end{Rem}

\begin{Rem}\label{rem:smoothness_away_from_equil}
	The proofs of Theorems~\ref{th:local-lin} and~\ref{th:global-lin} require only that the vector field $f$ is $C^k$ on the complement of $\{x_*\}$ and the flow $\Phi$ is $C^k$ on $\R\times B\setminus \{x_*\}$.
\end{Rem}

The final result establishes the mentioned relationship to the $4$-dimensional smooth Poincar\'{e} conjecture.
This conjecture asserts that every $4$-dimensional $C^\infty$ manifold homotopy equivalent to the $4$-sphere is diffeomorphic to the $4$-sphere.

\begin{Prop}\label{prop:poincare}
	Fix $k\in \N_{\geq 1}\cup \{\infty\}$.
	The $4$-dimensional smooth Poincar\'{e} conjecture is true if and only if the $C^k$ statement of Theorem~\ref{th:local-lin} (or Theorem~\ref{th:global-lin}) is true for $n=5$.
\end{Prop}

\section{Proofs}\label{sec:proofs}

We first assume Theorem~\ref{th:global-lin} to give a short proof of Theorem~\ref{th:local-lin}.

\begin{proof}[Proof of Theorem~\ref{th:local-lin}]
	Fix any Hurwitz matrix $A\in \R^{n\times n}$.
	Let $\psi\colon M \to [0,\infty)$ be a $C^\infty$ function equal to $1$ on a neighborhood $U_0$ of $x_*$ and equal to zero outside of a compact subset of $M$.
	Then $\psi f$ is a complete uniquely integrable continuous vector field generating a flow $\Phi$,  $x_*$ is asymptotically stable for $\Phi$, and $\Phi\in C^k$ if $f\in C^k$.
	According to Wilson, there is a proper strict $C^\infty$ Lyapunov function $V\colon B\to [0,\infty)$ for $x_*$ and $\Phi$ \cite[Thm~3.2]{wilson1969smooth} (see also \cite[Sec.~6]{fathi2019smoothing}).
	Theorem~\ref{th:global-lin} furnishes a homeomorphism $h_0\colon B\to \R^n$ that satisfies \eqref{eq:th:global-lin} and, if $n\neq 5$, restricts to a $C^k$-diffeomorphism $B\setminus \{x_*\}\to \R^n\setminus \{0\}$.
	Let $c>0$ be sufficiently small that $U\coloneqq V^{-1}([0,c))$ is contained in $U_0\cap B$.
	Then $h\coloneqq h_0|_U\colon U\to \R^n$ is the desired embedding.
\end{proof}

We now prove Theorem~\ref{th:global-lin} (without assuming Theorem~\ref{th:local-lin}).

\begin{proof}[Proof of Theorem~\ref{th:global-lin}]
	It suffices to find $h\colon B\to \R^n$ satisfying \eqref{eq:th:global-lin} for $A=-I$.
	Indeed, repeating the proof with $B$ and $f$ replaced by $\R^n$ and $\tilde{f}(x)=Ax$ then yields a corresponding transformation $\tilde{h}\colon \R^n\to \R^n$, so the composition $\tilde{h}^{-1}\circ h\colon B\to \R^n$  satisfies \eqref{eq:th:global-lin} for general $A$.
	
	Let  $V\colon B\to [0,\infty)$ be a proper strict $C^\infty$ Lyapunov function for $x_*$ and $\Phi$ (\cite[Thm~3.2]{wilson1969smooth}, \cite[Sec.~6]{fathi2019smoothing}).
	Fix any $c>0$ and note that $V^{-1}([0,c])$ is a contractible compact $C^\infty$ embedded submanifold with codimension-$1$ boundary $L\coloneqq V^{-1}(c)$ homotopy equivalent to the $(n-1)$-sphere $\Sph^{n-1}\coloneqq \{y\in \R^n\colon \|y\|=1\}$ \cite[pp.~326--327]{wilson1967structure}.
	Thus, if $n\neq 5$, there is a $C^\infty$-diffeomorphism $P\colon L\to \Sph^{n-1}$ according to classical facts for $n=2$ \cite[Appendix]{milnor1997topology} and $n=3$  \cite[Thm~9.3.11]{hirsch1994differential}, Perelman for $n=4$ \cite{perelman2002entropy,perelman2003ricci,perelman2003finite} (see \cite[Cor.~0.2]{morgan2007ricci}), and Smale for $n\geq 6$ \cite[Thm~5.1]{smale1962structure}.
	If $n=5$ there is still a homeomorphism $P\colon L\to \Sph^{n-1}$ according to Freedman \cite[Thm~1.6]{freedman1982topology}.
	
	Since for each $x\in B\setminus \{x_*\}$ the trajectory $t\mapsto \Phi^t(x)$  converges to $x_*$ and crosses $L$ exactly once and transversely, the map $\R\times L\to B \setminus \{x_*\}$ defined by $(t,x)\mapsto \Phi^t(x)$ is a homeomorphism and $C^k$-diffeomorphism if $\Phi\in C^k$, with inverse $g=(\tau,\rho)\colon B\setminus \{x_*\}\to \R\times L$ satisfying $\tau(x)\to -\infty$ as $x\to x_*$ (cf. \cite[p.~327]{wilson1967structure}).
	Thus $h\colon B\to \R^n$ defined by $h(x_*)=0$ and
	\begin{equation*}
		h(x)=e^{\tau(x)}P( \rho(x))
	\end{equation*}
	is a homeomorphism that, if $n\neq 5$ and $\Phi\in C^k$, restricts to a $C^k$-diffeomorphism $B\setminus \{x_*\}\to \R^n\setminus \{0\}$. 
	Finally, $h$ satisfies \eqref{eq:th:global-lin} with $A=-I_{n\times n}$ since $\rho\circ \Phi^t|_{B\setminus \{x_*\}}=\rho$ and $\tau \circ \Phi^t|_{B\setminus \{x_*\}}=\tau - t$.
\end{proof}

Finally, we prove Proposition~\ref{prop:poincare}.

\begin{proof}[Proof of Proposition~\ref{prop:poincare}]
	Assume that the $4$-dimensional smooth Poincar\'{e} conjecture is true.
	Then $P\colon L\to \Sph^{4}$ in the proof of Theorem~\ref{th:global-lin} can be taken to be a diffeomorphism when $n=5$.
	With this fact, repeating the proofs of Theorem~\ref{th:global-lin} and Theorem~\ref{th:local-lin} verbatim show that their $C^k$ statements are true for $n=5$.
	
	Conversely, assume that the $C^k$ statement of Theorem~\ref{th:local-lin} (or Theorem~\ref{th:global-lin}) is true for $n=5$ and $A=-I_{n\times n}$.
	Let $L$ be any $4$-dimensional $C^\infty$ homotopy sphere.
	According to Hirsch \cite[Thm~2]{hirsch1965homotopy}, there is a $C^\infty$ function $V\colon \Sph^5\to [0,1]$ having only two critical points, a maximum $N=V^{-1}(1)$ and minimum $S=V^{-1}(0)$, such that there is a diffeomorphism $P_c\colon L\to V^{-1}(c)$ for any $c\in (0,1)$.
	
	Equip $\Sph^5$ with a Riemannian metric and let $\Phi$ be the flow of the complete $C^\infty$ vector field $f=-\nabla V$ on $\Sph^5$.
	Observe that $S\in \Sph^5$ is asymptotically stable with basin of attraction $B\coloneqq \Sph^5\setminus \{N\}$ and that $f$ points inward at the boundary $V^{-1}(c)$ of $V^{-1}([0,c])$ for any $c\in (0,1)$.
	Thus, Theorem~\ref{th:local-lin} (or Theorem~\ref{th:global-lin})  furnishes an open neighborhood $U\subset \Sph^5$ of $S$ and a $C^k$ embedding $h\colon U\to \R^5$ such that, for any $c>0$ small enough that $V^{-1}(c)\subset U$, $h$ diffeomorphically maps $V^{-1}(c)$ onto a $C^k$ embedded submanifold $L_c\subset \R^5$ intersecting every line through the origin in a single point and transversely.
	Fixing such a small $c>0$ and letting $\rho\colon \R^5\setminus \{0\}\to \Sph^{4}$ be the straight line retraction $\rho(y)=\frac{y}{\|y\|}$, it follows that
	\begin{equation*}
		L\xrightarrow{P_c} V^{-1}(c) \xrightarrow{h}L_c  \xrightarrow{\rho} \Sph^{4}
	\end{equation*}
	is a well-defined $C^k$-diffeomorphism. 
	Since $L$ is a $C^\infty$ manifold $C^k$-diffeomorphic to $\Sph^4$, $L$ is $C^\infty$-diffeomorphic to $\Sph^4$ \cite[Thm~2.2.7]{hirsch1994differential}. 
\end{proof}

\section*{Additional Comments}

The paper \cite{grune1999asymptotic} dealt, more generally, with systems $\dot x = f(x,u)$, where $u$ denotes an input or disturbance. That paper showed that the input to state stability property \cite{mct} is equivalent, under similar global changes of variables, to finiteness of the $L^2$ norm (``$H_\infty$ gain'') of the input/state operator $(x(0),u(\cdot))\mapsto x(\cdot)$. As remarked in \cite{grune1999asymptotic}, however, the generalization of linearization constructions to systems with disturbances is not immediate. We thus leave open the study of such extensions.

We also note the relationship between the work reported here and areas of current machine learning research. In \cite[Chapter 6]{bramburger} one finds a discussion of ``autoencoder'' deep neural networks for the  numerical approximation of Hartman-Grobman-like conjugacies for linearizations and Koopman eigenvalues. Our results, especially in combination with new theoretical results about the existence of such autoencoders \cite{kvalheim_sontag_2023autoencoders}, contributes to the theoretical foundation for such studies.

Finally, the proof of Theorem~\ref{th:global-lin} is closely related to the proof of \cite[Prop.~1]{grune1999asymptotic}, which can be viewed as a (global) extension of the Morse lemma \cite{morse1932calculus} (see \cite[Lem.~2.2]{milnor1963morse}) to local minima of non-Morse functions.
Incorporating Perelman's result as in the proof of Theorem~\ref{th:global-lin}, the proof of \cite[Prop.~1]{grune1999asymptotic} can otherwise be repeated verbatim to prove Proposition~\ref{prop:gen-morse} below.
It extends the two-part statement \cite[Prop.~1]{grune1999asymptotic} by removing the hypothesis ``$n\neq 4$'' from the first part and relaxing the hypothesis ``$n\neq 4,5$'' from the second part to ``$n\neq 5$'', which reflects Perelman's result.
(It also contains the superficial extension of replacing the domain $\R^n$ of $V$ with a $C^\infty$ manifold $M$, but the hypotheses imply that $M$ is diffeomorphic to $\R^n$.)
Proposition~\ref{prop:gen-morse} is a global statement, but it readily implies a local statement in a manner similar to the implication of Theorem~\ref{th:local-lin} by Theorem~\ref{th:global-lin}.

For the following statement, a class $\mathcal{K}_\infty$ function is a strictly increasing and continuous function $\gamma\colon [0,\infty)\to [0,\infty)$ satisfying $\lim_{s\to\infty}\gamma(s)=\infty$.

\begin{Prop}\label{prop:gen-morse}
	Let $x_*$ be the unique critical point of a proper $C^1$ function $V\colon M\to \R$ on a connected $n$-dimensional $C^\infty$ manifold $M$. 
	Assume furthermore that $V$ is $C^\infty$ on $M\setminus \{x_*\}$.
	Then for each class $\mathcal{K}_\infty$ function $\gamma$ that is $C^\infty$ on $(0,\infty)$, there exists a homeomorphism $T\colon M\to \R^n$ with $T(x_*)=0$ such that
	\begin{equation*}
		V\circ T^{-1}(y)=\gamma(\|y\|).
	\end{equation*}
	In particular this holds for $\gamma(\|y\|)=\|y\|^2/2$.
	
	If $n\neq 5$ then $T$ can be chosen to restrict to a $C^\infty$-diffeomorphism  $M\setminus \{x_*\} \to \R^n\setminus \{0\}$.
	Furthermore, in this case there exists a class $\mathcal{K}_\infty$ function $\gamma$ which is $C^\infty$ on $(0,\infty)$ and satisfies $\gamma(s)/\gamma'(s)\geq s$ such that $T$ is $C^1$ with $DT(0)=0$.
\end{Prop}

\section*{Acknowledgments} This material is based upon work supported by the Air Force Office of Scientific Research under award number FA9550-24-1-0299 to Kvalheim and award number FA9550-21-1-0289 to Sontag.

\bibliographystyle{amsalpha}
\bibliography{global_linearization_kvalheim_sontag}

\providecommand{\bysame}{\leavevmode\hbox to3em{\hrulefill}\thinspace}
\providecommand{\MR}{\relax\ifhmode\unskip\space\fi MR }
\providecommand{\MRhref}[2]{%
  \href{http://www.ams.org/mathscinet-getitem?mr=#1}{#2}
}
\providecommand{\href}[2]{#2}
\begin{thebibliography}{BBKK22}

\bibitem[AK23]{arathoon2023koopman}
P~Arathoon and M~D Kvalheim, \emph{Koopman embedding and super-linearization
  counterexamples with isolated equilibria}, arXiv preprint arXiv:2306.15126
  (2023), 1--7.

\bibitem[Arn73]{Arnold92}
V~I Arnold, \emph{Ordinary differential equations}, The MIT Press, 1973,
  Translated and Edited by Richard A. Silverman.

\bibitem[BBKK22]{brunton2022modern}
S~L Brunton, M~Budi\v{s}i\'{c}, E~Kaiser, and J~N Kutz, \emph{Modern {K}oopman
  theory for dynamical systems}, SIAM Rev. \textbf{64} (2022), no.~2, 229--340.
  \MR{4416982}

\bibitem[BC23]{belabbas2023super}
M-A Belabbas and X~Chen, \emph{A sufficient condition for the
  super-linearization of polynomial systems}, Systems Control Lett.
  \textbf{179} (2023), Paper No. 105588, 6. \MR{4624015}

\bibitem[BCF12]{barta2012lyapunov}
T~B\'arta, R~Chill, and E~Fa\v{s}angov\'a, \emph{Every ordinary differential
  equation with a strict {L}yapunov function is a gradient system}, Monatsh.
  Math. \textbf{166} (2012), no.~1, 57--72. \MR{2901252}

\bibitem[BMM12]{budisic2012appliedkoopmanism}
M~Budi\v{s}i\'{c}, R~Mohr, and I~Mezi\'{c}, \emph{Applied {K}oopmanism}, Chaos
  \textbf{22} (2012), no.~4, 047510, 33. \MR{3388723}

\bibitem[Bra24]{bramburger}
J~J Bramburger, \emph{Data-driven methods for dynamic systems}, Society for
  Industrial and Applied Mathematics, Philadelphia, PA, 2024.

\bibitem[BS70]{BhatSzeg70}
N~Bhatia and G~Szeg{\"o}, \emph{Stability theory of dynamical systems, number
  161 in grundlehren der mathematischen wissenschaften}, Springer-Verlag, 1970.

\bibitem[Col65]{coleman1965local}
C~Coleman, \emph{Local trajectory equivalence of differential systems}, Proc.
  Amer. Math. Soc. \textbf{16} (1965), 890--892. \MR{180752}

\bibitem[Col66]{coleman1966addendum}
\bysame, \emph{Addendum to: {L}ocal trajectory equivalence of differential
  systems}, Proc. Amer. Math. Soc. \textbf{17} (1966), 770. \MR{192135}

\bibitem[EKR18]{eldering2018global}
J~Eldering, M~Kvalheim, and S~Revzen, \emph{Global linearization and fiber
  bundle structure of invariant manifolds}, Nonlinearity \textbf{31} (2018),
  no.~9, 4202--4245. \MR{3841342}

\bibitem[FP19]{fathi2019smoothing}
A~Fathi and P~Pageault, \emph{Smoothing {L}yapunov functions}, Trans. Amer.
  Math. Soc. \textbf{371} (2019), no.~3, 1677--1700. \MR{3894031}

\bibitem[Fre82]{freedman1982topology}
M~H Freedman, \emph{The topology of four-dimensional manifolds}, J.
  Differential Geometry \textbf{17} (1982), no.~3, 357--453. \MR{679066}

\bibitem[Gro59]{grobman1959homeomorphism}
D~M Grobman, \emph{Homeomorphism of systems of differential equations}, Doklady
  Akademii Nauk SSSR \textbf{128} (1959), no.~5, 880--881.

\bibitem[GSW99]{grune1999asymptotic}
L~Gr{\"u}ne, E~D Sontag, and F~R Wirth, \emph{Asymptotic stability equals
  exponential stability, and {ISS} equals finite energy gain—if you twist
  your eyes}, Systems \& Control Letters \textbf{38} (1999), no.~2, 127--134.

\bibitem[Har60]{hartman1960lemma}
P~Hartman, \emph{A lemma in the theory of structural stability of differential
  equations}, Proceedings of the American Mathematical Society \textbf{11}
  (1960), no.~4, 610--620.

\bibitem[Hir65]{hirsch1965homotopy}
M~W Hirsch, \emph{On homotopy spheres of low dimension}, Differential and
  {C}ombinatorial {T}opology ({A} {S}ymposium in {H}onor of {M}arston {M}orse),
  Princeton Univ. Press, Princeton, NJ, 1965, pp.~199--204. \MR{179800}

\bibitem[Hir94]{hirsch1994differential}
\bysame, \emph{Differential topology}, Graduate Texts in Mathematics, vol.~33,
  Springer-Verlag, New York, 1994, Corrected reprint of the 1976 original.
  \MR{1336822}

\bibitem[HK24]{haller2024data}
G~Haller and B~Kasz{\'a}s, \emph{Data-driven linearization of dynamical
  systems}, Nonlinear Dynamics \textbf{112} (2024), no.~21, 18639--18663.

\bibitem[Jon24]{jongeneel2024asymptotic}
W~Jongeneel, \emph{Asymptotic stability equals exponential stability---while
  you twist your eyes}, arXiv preprint arXiv:2411.03277 (2024), 1--27.

\bibitem[KA24]{kvalheim2024linearizability}
M~D Kvalheim and P~Arathoon, \emph{Linearizability of flows by embeddings},
  arXiv preprint arXiv:2305.18288v6 (2024), 1--20.

\bibitem[KB24]{ko2024minimum}
J~Ko and M-A Belabbas, \emph{Minimum number of observables and system
  invariants in super-linearization}, IEEE Control Syst. Lett. \textbf{8}
  (2024), 2217--2222. \MR{4816989}

\bibitem[KR21]{kvalheim2021existence}
M~D Kvalheim and S~Revzen, \emph{Existence and uniqueness of global {K}oopman
  eigenfunctions for stable fixed points and periodic orbits}, Phys. D
  \textbf{425} (2021), Paper No. 132959, 20. \MR{4275046}

\bibitem[KS24]{kvalheim_sontag_2023autoencoders}
M~D Kvalheim and E~D Sontag, \emph{Why should autoencoders work?}, Transactions
  on Machine Learning Research (2024).

\bibitem[LM13]{lan2013linearization}
Y~Lan and I~Mezi\'{c}, \emph{Linearization in the large of nonlinear systems
  and {K}oopman operator spectrum}, Phys. D \textbf{242} (2013), 42--53.
  \MR{3001394}

\bibitem[LOS23]{liu2023non}
Z~Liu, N~Ozay, and E~D Sontag, \emph{On the non-existence of immersions for
  systems with multiple omega-limit sets}, IFAC World Congress, Yokohoma, Japan
  \textbf{56} (2023), 60--64.

\bibitem[LOS25]{liu2024properties}
\bysame, \emph{Properties of immersions for systems with multiple limit sets
  with implications to learning {K}oopman embeddings}, Automatica (2025), To
  appear. Preprint can be found in https://arxiv.org/abs/2312.17045, 2023/2024.

\bibitem[Mez20]{mezic2020spectrum}
I~Mezi\'c, \emph{Spectrum of the {K}oopman operator, spectral expansions in
  functional spaces, and state-space geometry}, J. Nonlinear Sci. \textbf{30}
  (2020), no.~5, 2091--2145. \MR{4163461}

\bibitem[Mil63]{milnor1963morse}
J~Milnor, \emph{Morse theory}, Annals of Mathematics Studies, vol. No. 51,
  Princeton University Press, Princeton, NJ, 1963, Based on lecture notes by M.
  Spivak and R. Wells. \MR{163331}

\bibitem[Mil97]{milnor1997topology}
J~W Milnor, \emph{Topology from the differentiable viewpoint}, Princeton
  Landmarks in Mathematics, Princeton University Press, Princeton, NJ, 1997,
  Based on notes by David W. Weaver, Revised reprint of the 1965 original.
  \MR{1487640}

\bibitem[Mor96]{morse1932calculus}
M~Morse, \emph{The calculus of variations in the large}, American Mathematical
  Society Colloquium Publications, vol.~18, American Mathematical Society,
  Providence, RI, 1996, Reprint of the 1932 original. \MR{1451874}

\bibitem[MT07]{morgan2007ricci}
J~Morgan and G~Tian, \emph{{R}icci flow and the {P}oincar\'e{} conjecture},
  Clay Mathematics Monographs, vol.~3, American Mathematical Society,
  Providence, RI; Clay Mathematics Institute, Cambridge, MA, 2007. \MR{2334563}

\bibitem[Per02]{perelman2002entropy}
G~Perelman, \emph{The entropy formula for the {R}icci flow and its geometric
  applications}, arXiv preprint math/0211159 (2002), 1--39.

\bibitem[Per03a]{perelman2003finite}
\bysame, \emph{Finite extinction time for the solutions to the {R}icci flow on
  certain three-manifolds}, arXiv preprint math/0307245 (2003), 1--7.

\bibitem[Per03b]{perelman2003ricci}
\bysame, \emph{Ricci flow with surgery on three-manifolds}, arXiv preprint
  math/0303109 (2003), 1--22.

\bibitem[Sma62]{smale1962structure}
S~Smale, \emph{On the structure of manifolds}, Amer. J. Math. \textbf{84}
  (1962), 387--399. \MR{153022}

\bibitem[Son98]{mct}
E~D Sontag, \emph{Mathematical {C}ontrol {T}heory. {D}eterministic
  {F}inite-{D}imensional {S}ystems}, second ed., Texts in Applied Mathematics,
  vol.~6, Springer-Verlag, New York, 1998.

\bibitem[Son10]{sontag10yamamoto}
E~D Sontag, \emph{Contractive systems with inputs}, Perspectives in
  Mathematical System Theory, Control, and Signal Processing (J.~Willems,
  S.~Hara, Y.~Ohta, and H.~Fujioka, eds.), Springer-verlag, 2010, pp.~217--228.

\bibitem[Wil67]{wilson1967structure}
F~W Wilson, Jr, \emph{The structure of the level surfaces of a {L}yapunov
  function}, J. Differential Equations \textbf{3} (1967), 323--329.
  \MR{0231409}

\bibitem[Wil69]{wilson1969smooth}
\bysame, \emph{Smoothing derivatives of functions and applications}, Trans.
  Amer. Math. Soc. \textbf{139} (1969), 413--428. \MR{0251747}

\bibitem[Wil78]{wilson1978reformulation}
\bysame, \emph{A reformulation of {C}oleman's conjecture concerning the local
  conjugacy of topologically hyperbolic singular points}, The structure of
  attractors in dynamical systems ({P}roc. {C}onf., {N}orth {D}akota {S}tate
  {U}niv., {F}argo, {N}.{D}., 1977), Lecture Notes in Math., vol. 668,
  Springer, Berlin-New York, 1978, pp.~245--252. \MR{518564}

\bibitem[WKR15]{williams2015data}
M~O Williams, I~G Kevrekidis, and C~W Rowley, \emph{A data--driven
  approximation of the {K}oopman operator: Extending dynamic mode
  decomposition}, Journal of Nonlinear Science \textbf{25} (2015), 1307--1346.

\end{thebibliography}

\end{document}